\documentclass[a4paper]{article}
\usepackage{amsthm,amssymb,amsmath,enumerate}
\usepackage{algorithm}

\usepackage{tikz}
\usetikzlibrary{backgrounds}
\usetikzlibrary{decorations}
\usetikzlibrary{decorations.pathmorphing}

\tikzstyle{vx}=[thick,circle,inner sep=0.cm, minimum size=2mm, fill=white, draw=black]
\tikzstyle{edg}=[very thick]
\tikzstyle{rededge}=[very thick,red]
\tikzstyle{point}=[draw,circle,inner sep=0.cm, minimum size=1mm, fill=black]
\tikzstyle{pointer}=[thick,->,shorten >=2pt,color=dunkelgrau]
\tikzstyle{facebdry}=[color=auchblau, very thick] % face boundary
\tikzstyle{face}=[facebdry,fill=hellblau]
\tikzstyle{nface}=[color=hellblau,fill=hellblau,thick] % naked face, without boundary
\tikzstyle{tinyvx}=[thick,circle,inner sep=0.cm, minimum size=1.3mm, fill=white, draw=black]
\tikzstyle{tinyA}=[thick,circle,inner sep=0.cm, minimum size=1.3mm, fill=hellgrau, draw=black]

\tikzstyle{patedg}=[edg,decorate, decoration={random steps,segment length=3pt,amplitude=1pt}]

\colorlet{auchblau}{blue!60!white}
\colorlet{hellblau}{blue!20!white}
\colorlet{hellrot}{red!40!white}
\colorlet{hellgrau}{black!30!white}
\colorlet{grauish}{black!45!white}
\colorlet{dunkelgrau}{black!60!white}

\newcommand{\block}[1]{
\begin{scope}[shift={#1}]
\draw[line width=2pt,dunkelgrau,fill=hellgrau] (0,0) circle [radius=\crad];
\end{scope}
}

\newtheorem{definition}{Definition}
\newtheorem{proposition}[definition]{Proposition}
\newtheorem{theorem}[definition]{Theorem}

\newtheorem{lemma}[definition]{Lemma}

\newcommand{\comment}[1]{}

\newcommand{\emtext}[1]{\text{\em #1}}

\newcommand{\bigO}{O}

\title{Even $A$-cycles have the edge-Erd\H os-P\'osa property}
\author{Henning Bruhn\thanks{Partially supported by DFG, grant no.\  321904558}}
\date{}

\begin{document}
\maketitle

\begin{abstract}
I prove that even $A$-cycles have the edge-Erd\H os-P\'osa property.
\end{abstract}

\section{Introduction}

Recently, it has become ever more clear that there is a significant
difference between the ordinary Erd\H os-P\'osa property and its
edge-version. Initially it might have seen as if packing and covering
was largely the same 
in the vertex world and in the edge
world: as Erd\H os and P\'osa~\cite{EP65} showed every graph either has $k$ 
\emph{vertex-disjoint} cycles or a \emph{vertex set} of size $\bigO(k\log k)$ meeting all cycles,
and similarly, there are always either $k$ \emph{edge-disjoint} cycles 
or an \emph{edge set} of size $\bigO(k\log k)$ meeting all cycles. 
The same is true for many other classes of target objects. 

Say that a class $\mathcal F$ of graphs has the (ordinary) \emph{Erd\H os-P\'osa property}
(resp.\ the \emph{edge-Erd\H os-P\'osa property})
if there is a function $f$ such that for every positive integer $k$, every graph $G$
either contains $k$ disjoint (resp.\ edge-disjoint) subgraphs each isomorphic to some graph in $\mathcal F$,
or if there is a vertex set $X$ (resp.\ an edge set) of size $|X|\leq f(k)$ such that $G-X$ is devoid
of subgraphs from $\mathcal F$. 
Then not only have cycles  both the ordinary and the edge-Erd\H os-P\'osa property but this 
is also true for even cycles~\cite{DNL87,Tho88,BHJ18}; for $A$-cycles\footnote{Technically,
the definition of the Erd\H os-P\'osa property only applies to graph classes but not to 
such objects as $A$-cycles, which could be seen as  graphs with labels. 
I assume the property to extend to $A$-cycles and similar objects in the obvious way.  
}~\cite{KKM11,PW12}, 
cycles that each contain at least one vertex from 
some fixed set $A$; long cycles~\cite{RS86,BHJ17}, cycles that have a certain pre-fixed minimum length; 
$K_4$-subdivisions~\cite{RS86,BH18};
and many other classes of graphs. 

Recently, however, differences have been discovered. While I do not know of any 
 class that has the edge-Erd\H os-P\'osa property but not the ordinary property, the converse
does exist: even $A$-paths have the ordinary property but not the edge-property~\cite{BHJ18}. The same 
is true for subdivisions of subcubic trees of large pathwidth~\cite{BHJ18b}. 
%In this article, 
I will present a class that sits, in some sense, right at the edge: even $A$-cycles
turn out to have both properties, the ordinary and the edge-Erd\H os-P\'osa property, but 
slightly modifying the constraints breaks the edge-property but not the ordinary one.

That even $A$-cycles have the ordinary Erd\H os-P\'osa property is a special case of a much more 
general result due to Kakimura and Kawarabayashi:

\begin{theorem}[Kakimura and Kawarabayashi~\cite{KK12}]\label{kkthm}
Let $m>1$ be an integer. Then
$A$-cycles of length divisible by $m$ have the Erd\H os-P\'osa property.
\end{theorem}
%Kakimura and Kawarabayashi do not specify a size of the hitting set. 

I will prove:
\begin{theorem}\label{evenathm}
Even $A$-cycles have the edge-Erd\H os-P\'osa property.
\end{theorem}

Allowing more complicated constraints on the length of the cycles, as in the theorem 
of Kakimura and Kawarabayashi, breaks the edge-property. $A$-cycles of a length divisible by~$3$,
for example, do not have the edge-Erd\H os-P\'osa property; see Section~\ref{secbreak}.
We can also insist on a certain minimum length, i.e.\ only consider even $A$-cycles of a length
at least~$\ell$. Imposing a minimum length never seems to affect the ordinary Erd\H os-P\'osa
property, and indeed long even $A$-cycles still have the ordinary property~\cite{BJS18X}. In 
contrast, the edge-property is lost; see Section~\ref{secbreak}. 

How large is the \emph{hitting set} in Theorem~\ref{evenathm}? 
(The hitting set is the vertex or edge set that meets all subgraphs from the target class.)
 Unfortunately, 
the size  depends on the vertex hitting set of Theorem~\ref{kkthm}, and Kakimura and 
Kawarabayashi do not give any explicit bound on their hitting set.
The bound  is likely very large %, however,  
as Kakimura and Kawarabayashi  extensively use the techniques of minor theory, where bounds tend to become huge. 
It is possible, however, to say what the proof of Theorem~\ref{evenathm} adds. If $g(k)$
is an upper bound on the size of a vertex hitting set for even $A$-cycles then 
the edge hitting set of Theorem~\ref{evenathm} has size at most
 $1080k^5\cdot g(k)$. Clearly, this is far from optimal. 

\medskip

There is an extensively list of result on the ordinary Erd\H os-P\'osa property, 
see for instance~\cite{RT17} and~\cite{BHJ18} for  references. There is far less known about
the edge-property. Even cycles~\cite{BHJ18} as well as $A$-cycles~\cite{PW12} were known to have the edge-property.
For other simple classes, such as the subdivisions of a ladder graph with four rungs or of
the binary tree of height~$3$, it is, on the other hand, still open whether they possess the edge-Erd\H os-P\'osa property.

\section{More constraints break the edge-property}\label{secbreak}

Before proving the main theorem I will first briefly discuss how extensions of Theorem~\ref{evenathm} fail.
In the  notation I follow Diestel~\cite{diestelBook17}, where also basic graph-theoretic concepts may
be found. All graphs will be simple and finite.

\begin{proposition}
Let $\ell\geq 5$ be an integer. 
Even $A$-cycles of length at least~$\ell$
 do not have the edge-Erd\H os-P\'osa property.
\end{proposition}

Fix an $\ell\geq 5$, and call a cycle of length at least~$\ell$ \emph{long}. 
We construct for every integer $h>0$ a graph $G_h$ that does not contain two edge-disjoint 
long even $A$-cycles and that does not admit an edge set of size at most~$h$ that meets
all long even $A$-cycles. This then shows that long even $A$-cycles do not 
have the edge-Erd\H os-P\'osa property. 

To construct $G_h$ we start with an (elementary) wall $W$ of size $10h\times 10h$. 
(See for instance~\cite{BJS18X} for a formal 
definition of a wall, or Figure~\ref{Acycfig} for a picture of a $6\times 6$-wall.)
Let $U,V$ be the two bipartition classes of the bipartite wall $W$.
We add a complete bipartite graph with one bipartition class consisting of two vertices, $u$ and $v$,
and the other consisting of the vertex set $A$ of size~$10h$. We link $u$ to the left-most
vertex in every second row of the wall $W$, and $v$ to the right-most vertex in every second
row such that the neighbours of $u$ in $W$ are all in $U$ and the neighbours of $v$ in $W$ are all in $V$.
Finally, we suppress the vertices of $V$ in the top row such that only the vertices of $V$
remain in the top row; see Figure~\ref{Acycfig}. (\emph{Suppressing} a vertex of degree~$2$ means
to replace it by an edge between its neighbours.)

We first observe that any path in $W$ between a neighbour of $u$ and a neighbour of $v$ 
has odd length, unless it passes through an odd number of edges from the top row. Thus, any even $A$-cycle
that meets the wall needs to traverse the wall from left to right, from a neighbour of $u$ to a neighbour
of $v$, and to pass through at least one edge from the top row. Clearly, there cannot  
be two such edge-disjoint cycles. On the other hand, any cycle of length at least~$5$ has to meet the wall.

To see that there is no set of at most $h$ edges that meets all 
long even $A$-cycles, it suffices to observe that $u$ and $v$ are well connected to the wall (each with 
$5h$ edges), that the wall is sufficiently large and that it contains many parity breaking edges, the 
edges in the top row.

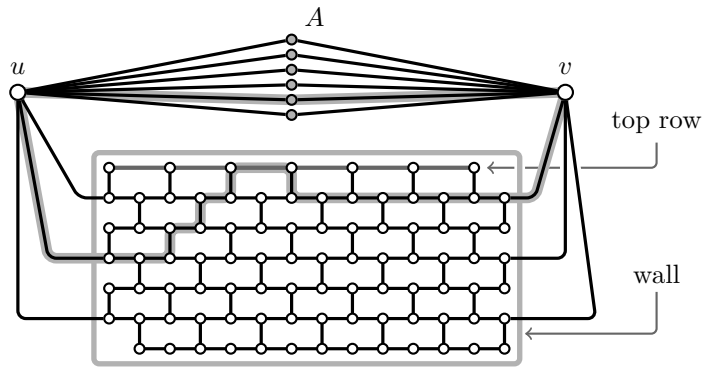
\begin{figure}[ht]
\centering
\begin{tikzpicture}
\tikzstyle{tinyvx}=[thick,circle,inner sep=0.cm, minimum size=1.3mm, fill=white, draw=black]

\def\wallheight{6}
\def\brickheight{0.4}

\pgfmathtruncatemacro{\lastrow}{\wallheight}
\pgfmathtruncatemacro{\penultimaterow}{\wallheight-1}
\pgfmathtruncatemacro{\lastrowshift}{mod(\wallheight,2)}
\pgfmathtruncatemacro{\lastx}{2*\wallheight+1}

% frame around wall
\def\offset{0.2}
\draw[line width=2pt,rounded corners=2pt, hellgrau] (0-\offset,0-\offset) rectangle (2*\wallheight*\brickheight+\brickheight+\offset,\wallheight*\brickheight+\offset);

\node (W) at (2*\wallheight*\brickheight+\brickheight+2,1) {wall};
\draw[pointer, rounded corners=2pt] (W) to (2*\wallheight*\brickheight+\brickheight+2,0.5*\brickheight) to 
(2*\wallheight*\brickheight+\brickheight+0.2,0.5*\brickheight);

\node[align=center] (T) at (2*\wallheight*\brickheight+\brickheight+2,3) {top row};

\draw[edg] (\brickheight,0) -- (2*\wallheight*\brickheight+\brickheight,0);
\foreach \i in {1,...,\penultimaterow}{
  \draw[edg] (0,\i*\brickheight) -- (2*\wallheight*\brickheight+\brickheight,\i*\brickheight);
}

% last row!
\draw[edg,dunkelgrau,ultra thick] (\lastrowshift*\brickheight,\lastrow*\brickheight) to ++(2*\wallheight*\brickheight,0);

\foreach \j in {0,2,...,\penultimaterow}{
  \foreach \i in {0,...,\wallheight}{
    \draw[edg] (2*\i*\brickheight+\brickheight,\j*\brickheight) to ++(0,\brickheight);
  }
}
\foreach \j in {1,3,...,\penultimaterow}{
  \foreach \i in {0,...,\wallheight}{
    \draw[edg] (2*\i*\brickheight,\j*\brickheight) to ++(0,\brickheight);
  }
}

\def\first{0}

\foreach \i in {1,...,\lastx}{
  \node[tinyvx] (w\i\first) at (\i*\brickheight,0){};
}
\foreach \j in {1,...,\penultimaterow}{
  \foreach \i in {0,...,\lastx}{
    \node[tinyvx] (w\i\j) at (\i*\brickheight,\j*\brickheight){};
  }
}
\foreach \i in {1,3,...,\lastx}{
  \node[tinyvx] (w\i\lastrow) at (\i*\brickheight+\lastrowshift*\brickheight-\brickheight,\lastrow*\brickheight){};
}

\node[vx,label=above:$u$] (A1) at (-\wallheight*0.5*\brickheight,\wallheight*\brickheight+1){};
\node[vx,label=above:$v$] (A2) at (\wallheight*\brickheight*2.5,\wallheight*\brickheight+1){};

\foreach \i in {1,...,\wallheight}{
  \node[tinyA] (a\i) at (\wallheight*\brickheight,\wallheight*\brickheight+\i*0.5*\brickheight+0.5){};
  \draw[edg] (A1) -- (a\i) -- (A2);
}

\def\hstep{0.2}

\foreach \j in {1,3,...,\penultimaterow}{
  \draw[edg,rounded corners=3pt] (w0\j) -- ++(-\wallheight*\hstep+\j*\hstep-\hstep,0) -- (A1); 
  \draw[edg,rounded corners=3pt] (w\lastx\j) -- ++(+\wallheight*\hstep-\j*\hstep+\hstep,0) -- (A2); 
}

\node at (\wallheight*\brickheight+0.3,\wallheight*\brickheight+\wallheight*0.5*\brickheight+0.8){$A$};

\begin{scope}[on background layer]
\draw[pointer, rounded corners=2pt] (T) to (2*\wallheight*\brickheight+\brickheight+2,\wallheight*\brickheight) to 
(2*\wallheight*\brickheight+0.1,\wallheight*\brickheight);

\draw[line width=4pt,rounded corners=2pt, white] (0-\offset,0-\offset) rectangle (2*\wallheight*\brickheight+\brickheight+\offset,\wallheight*\brickheight+\offset);

\foreach \j in {1,3,...,\penultimaterow}{
  \draw[line width=4pt,white,rounded corners=3pt] (w\lastx\j) -- ++(+\wallheight*\hstep-\j*\hstep+\hstep,0) -- (A2); 
}

\draw[line width=4pt,hellgrau,rounded corners=3pt] (a2.center) -- (A1.center) -- (-2*\brickheight,3*\brickheight) --
(w03.center) -- ++(2*\brickheight,0) -- ++(0,\brickheight)
-- ++(\brickheight,0) -- ++(0,\brickheight) -- ++(\brickheight,0) -- ++(0,\brickheight) -- ++(2*\brickheight,0)
-- ++(0,-\brickheight) -- ++(8*\brickheight,0)  -- (A2.center) -- (a2.center);

\end{scope}
\end{tikzpicture}
\caption{Long even $A$-cycles do not have the edge-Erd\H os-P\'osa property. 
A long even $A$-cycle is shown in grey.}\label{Acycfig}
\end{figure}

Theorem~\ref{kkthm} of Kakimura and Kawarabayashi covers quite general modularity constraints on 
the length of the $A$-cycles. All of these, except for requiring the cycles to be even, break the
edge-Erd\H os-P\'osa property: 
\begin{proposition}\label{modprop}
For integers $m>2$, 
the class of $A$-cycles of length $\equiv 0\text{ mod }m$
 does not have the edge-Erd\H os-P\'osa property.
\end{proposition}

The construction is quite similar to the one for long even $A$-cycles; it suffices
to subdivide the edges in the graph $G_h$ shown in Figure~\ref{Acycfig} in a suitable way:
Subdivide all edges between $u$ and $A$ such that the
resulting paths have length~$m-2$, leave the edges between $v$ and $A$ and 
the edges in the top row (in grey) as they are, but subdivide all other edges so that they become
paths of length~$m$. Then, $A$-cycles of length $\equiv 0\text{ mod }m$ will have the same 
form as before: In particular, they will contain a subpath that traverses the wall from left to 
right and passes through an edge in the top row in between. Here, we use that $m\neq 2$ to 
exclude that the subdivision of the complete bipartite graph between $A$ and $\{u,v\}$ 
contains an $A$-cycle of length $\equiv 0\text{ mod }m$. Any cycle there has 
length $2(m-2)+2\cdot 1=2m-2\not\equiv 0$ whenever $m\neq 2$.

\section{Basic definitions and lemmas}

In the remainder of the article I will prove Theorem~\ref{evenathm}. In this 
section I collect simple lemmas and a little bit of notation.  
In particular, we will use Diestel's~\cite{diestelBook17} path notation. That is, if $P$ is a 
path that contains the vertices $u$ and $v$  then $uPv$
is the subpath of $P$ with endvertices $u$ and $v$.

A second central notion is that of a \emph{block-tree} of a connected graph $G$. 
If $\mathcal B$ is the set of blocks of $G$ and $C$ the set of cutvertices of $G$
then the block-tree of $G$ is defined on $\mathcal B\cup C$ as vertex set such that 
$Bc$ with $B\in\mathcal B$ and $c\in C$ is an edge precisely when $c$ lies in $B$.

A \emph{theta graph} is a subdivision of the multigraph
consisting of two vertices joined by three parallel edges. 
As we deal with (simple) graphs we insist that two of the
parallel edges need to be subdivided. The two vertices of degree~$3$
in a theta graph are its \emph{branch vertices}, the three internally disjoint
paths between the branch vertices are its \emph{subdivided edges}.

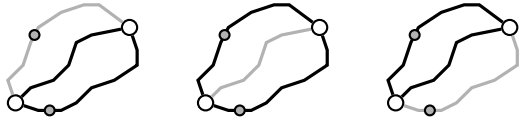
\begin{figure}[ht]
\centering
\begin{tikzpicture}
\node[vx] (u) at (0,0){};
\node[vx] (v) at (1.5,1){};
\draw[edg,hellgrau] (u) to ++(-0.1,0.3) to ++(0.2,0.2) to ++(0.1,0.3) to node[tinyA]{} ++(0.1,0.2) to ++(0.3,0.2) to ++(0.3,0.1) 
to ++(0.2,0) to (v);
\draw[edg] (u) to ++(0.2,0.2) to ++(0.3,0.1) to ++(0.2,0.2) to ++(0.1,0.3) to ++(0.2,0.1) 
to (v);
\draw[edg] (u) to ++(0.3,-0.1) to node[tinyA]{} ++(0.3,0) to ++(0.2,0.1) to ++(0.2,0.2) to ++(0.3,0.1) to ++(0.3,0.2) 
to ++(0,0.2) to (v);

\begin{scope}[shift={(2.5,0)}]
\node[vx] (u) at (0,0){};
\node[vx] (v) at (1.5,1){};
\draw[edg] (u) to ++(-0.1,0.3) to ++(0.2,0.2) to ++(0.1,0.3) to node[tinyA]{} ++(0.1,0.2) to ++(0.3,0.2) to ++(0.3,0.1) 
to ++(0.2,0) to (v);
\draw[edg,hellgrau] (u) to ++(0.2,0.2) to ++(0.3,0.1) to ++(0.2,0.2) to ++(0.1,0.3) to ++(0.2,0.1) 
to (v);
\draw[edg] (u) to ++(0.3,-0.1) to node[tinyA]{} ++(0.3,0) to ++(0.2,0.1) to ++(0.2,0.2) to ++(0.3,0.1) to ++(0.3,0.2) 
to ++(0,0.2) to (v);
\end{scope}

\begin{scope}[shift={(5,0)}]
\node[vx] (u) at (0,0){};
\node[vx] (v) at (1.5,1){};
\draw[edg] (u) to ++(-0.1,0.3) to ++(0.2,0.2) to ++(0.1,0.3) to node[tinyA]{} ++(0.1,0.2) to ++(0.3,0.2) to ++(0.3,0.1) 
to ++(0.2,0) to (v);
\draw[edg] (u) to ++(0.2,0.2) to ++(0.3,0.1) to ++(0.2,0.2) to ++(0.1,0.3) to ++(0.2,0.1) 
to (v);
\draw[edg,hellgrau] (u) to ++(0.3,-0.1) to node[tinyA]{} ++(0.3,0) to ++(0.2,0.1) to ++(0.2,0.2) to ++(0.3,0.1) to ++(0.3,0.2) 
to ++(0,0.2) to (v);
\end{scope}
\end{tikzpicture}
\caption{The theta graph in Lemma~\ref{thetalem}; vertices in $A$ in grey.}\label{thetafig}
\end{figure}

We need a very simple lemma that nevertheless is key to the proof. 
I contend that, in some sense, the reason why \emph{even} $A$-cycles
have the edge-Erd\H os-P\'osa property but not $A$-cycles of a length divisible
by~$3$, say, lies in the fact that theta graphs always contain an even cycle
but not necessarily one of a length divisible by~$3$.
\begin{lemma}\label{thetalem}
Let $\theta$ be a theta graph such that two of its subdivided edges meet 
a vertex set $A$. Then $\theta$ contains an even $A$-cycle.
\end{lemma}
\begin{proof}
Let $P_1,P_2,P_3$ be the subdivided edges of $\theta$. Then each of $P_1\cup P_2$,
$P_2\cup P_3$ and $P_3\cup P_1$ is a an $A$-cycle, and at least one of these is even.
\end{proof}

A second observation shows that we can exploit the fact that 
the ordinary Erd\H os-P\'osa property is already known to hold for even $A$-cycles.
Here, we call the upper bound on a hitting set, the function $f$ of the definition 
of the (ordinary or edge-) Erd\H os-P\'osa property,  a \emph{bounding function}.
\begin{lemma}[Bruhn and Heinlein~\cite{BH18}]\label{singlehitlem}
Let $\mathcal F$ be a class of graphs that has the Erd\H os-P\'osa property
with bounding function~$g$. Let $h:\mathbb N\to\mathbb R$ be a function such that 
for every $k$ and for every graph $G$ that has a vertex $z$ such that $G-z$ does not contain any subgraph of $\mathcal F$
it holds that: either $G$ contains $k$ edge-disjoint subgraphs from $\mathcal F$
or there is an edge set $F$ of size $|F|\leq h(k)$ that meets every subgraph from $\mathcal F$.
Then $\mathcal F$ has the edge-Erd\H os-P\'osa property with bounding function $f=gh$.
\end{lemma}
We note that, although not formulated in this way, the lemma and indeed its proof still hold true 
for classes of graphs with labels, such as even $A$-cycles.

We need a couple of more simple lemmas.
Let $a,b,c$ be three distinct vertices. An \emph{$a$--$b$--$c$~path} is a path 
that starts in $a$, passes through $b$ and ends in $c$.

\begin{lemma}\label{abclem}
Let $a,b,c$ be distinct vertices in a graph $G$. Then $G$ contains an $a$--$b$--$c$~path if and only 
if no single vertex, except for $b$, separates $b$ from $\{a,c\}$.
\end{lemma}
\begin{proof}
As no single vertex separates $b$ from $\{a,c\}$ there are two internally disjoint $b$--$\{a,c\}$~paths $P_1,P_2$. 
Clearly, we are done if $P_1$ and $P_2$ have different endvertices. Thus, we assume that $P_1$ and $P_2$ end in $c$.
Again, as no single vertex separates $b$ from $\{a,c\}$ there is a path $Q$ from $b$ to $a$ that avoids $c$. 
Viewed from $a$ let $z$ be the first vertex on $Q$ in $P_1\cup P_2$. Let us say that $z\in V(P_1)$.
 Then $aQzP_1bP_2c$ is an $a$--$b$--$c$~path.
\end{proof}

%\section{Tree lemmas}

The \emph{diameter} of a tree is the length of a longest path in the tree. 
\begin{lemma}\label{treesizelem}
Every tree with $s\geq 2$ leaves and diameter $d$ has at most $\tfrac{sd}{2}+1$  vertices.
\end{lemma}
% note: this is tight if $d$ is even. Subdivide a star with $s$ leaves such that each
% edge gets length $d/2$. 
\begin{proof}
Let $T$ be a tree with $s$ leaves and diameter $d$. For $s=2$ the tree $T$ is a path of length $d$
and thus has $d+1=\tfrac{sd}{2}+1$ many vertices. For $s\geq 3$, pick a longest path $P$
and a leaf $\ell$ not on $P$, and let $Q$ be the path that starts in $\ell$, ends in 
a vertex $t$ of degree at least~$3$ and that has no internal vertex of degree~$3$ or more. 
In particular, except for possibly $t$, no vertex in $Q$ lies in the longest path $P$.
Moreover, we note that the length of $Q$ is at most $d/2$. 

Let $T'$ be the tree obtained from deleting all of $Q$ except for $t$. Then  $T'$
has $s-1$ leaves and at most $d/2$ fewer vertices than $T$. Induction on $s$ now finishes the proof.
\end{proof}

Let $T$ be a rooted tree with root $r$. For any vertex $v$ we denote by $T_v$ the subtree 
of $T$ on all vertices $w$ for which the $w$--$r$~path in $T$ passes through $v$.

\begin{lemma}\label{3leaflem}
Every tree $T$ with $\Delta(T)\geq 3$ and with $s$ leaves contains 
$\lfloor s/(2\Delta(T))\rfloor$
%\[
%\left\lfloor\frac{s}{2\Delta(T)-2}\right\rfloor
%\]
disjoint subtrees that each contain three leaves of $T$.
\end{lemma}
\begin{proof}
Pick a root $r$ and consider $T$ as a rooted tree. In $T$ choose a vertex $v$
that is farthest from $r$ such that the subtree $T_v$ contains at least three leaves.
Then the components of $T_v-v$ each contain at most two leaves of $T$, and there are at most
$\Delta(T)$ such components. 
Therefore, $T_v$ contains at most $2\Delta(T)$ many 
leaves of $T$ (this is also true if $v$ is a leaf itself). 
We now delete $T_v$ from $T$, and then delete iteratively 
 each leaf that is not a leaf from $T$.  Then all leaves of the resulting tree $T'$ 
are leaves of $T$, and $T'$ has at least $s-2\Delta(T)$ many leaves. Induction 
on the number of leaves now yields the desired result.
\end{proof}

\begin{lemma}\label{alttreelem}
Let $\alpha,\beta$ and $\gamma$ be positive integers. 
For every tree $T$ on at least $\alpha\beta\gamma$ vertices  at least one of the
following statements is true:
\begin{enumerate}[\rm (i)]
\item\label{treea} $\Delta(T)\geq \alpha$; or
\item\label{treeb} $T$ contains $\beta$ disjoint subtrees that each contain three leaves of $T$; or
\item\label{treec} there is a path in $T$ of length at least $\gamma$.
\end{enumerate}
\end{lemma}
\begin{proof}
Assume $T$ to be a tree that satisfies neither \eqref{treea}, \eqref{treeb} nor \eqref{treec}.
We will bound the number $n$ of its vertices. 
By Lemma~\ref{3leaflem}, it follows for the number $s$ of leaves of $T$ that
\[
\left\lfloor\frac{s}{2\Delta(T)}\right\rfloor \leq \beta-1
\quad\Rightarrow\quad
\frac{s}{2\Delta(T)} \leq \beta,
\]
which implies $s\leq 2(\alpha-1)\beta\leq 2\alpha\beta-2$.

Lemma~\ref{treesizelem} yields
\(
n\leq\tfrac{1}{2}s(\gamma-1)+1< \alpha\beta\gamma.
\)
\end{proof}

The following lemma is a more general form of a lemma in~\cite{BHJ18}, which,
in turn, is the finite special case of a lemma of Thomassen~\cite{Tho16}. 
It can be proved with a simple greedy argument. 
\begin{lemma}\label{Asublem}
Let $s$ be a positive integer, and 
let $Z$ be a vertex set in a tree~$T$. Then $T$ contains $\lfloor |Z|/s\rfloor$ 
edge-disjoint subtrees that each contain at least $s$ vertices from $Z$. 
\end{lemma}

\section{Proof of main result}

In the course of this section we will prove that for all positive integers $k$, 
every graph $G$ with a single vertex hitting set
either contains $k$ edge-disjoint even $A$-cycles or admits an edge set $F$ of size $|F|\leq 1080k^5$
that meets all even $A$-cycles. Lemma~\ref{singlehitlem}, together with Theorem~\ref{kkthm}, 
 then finishes the proof of Theorem~\ref{evenathm}.

We fix throughout this section a graph $G$, a vertex set $A\subseteq V(G)$, and an integer $k\geq 2$. 
For the proof we can make three assumptions:
\begin{equation}\label{singlez}
\emtext{There is a vertex $z$ that meets every even $A$-cycle of $G$}.
\end{equation}
\begin{equation}\label{nok}
\emtext{$G$ does not contain $k$ edge-disjoint even $A$-cycles}.
\end{equation}
\begin{equation}\label{G2con}
\emtext{$G$ is $2$-connected}.
\end{equation}
The last assumption is justified because every even $A$-cycle lies in a block of $G$.

\begin{lemma}
If $z\in A$ then there is an edge set $F$ of size $|F|\leq 4k$ that meets all even $A$-cycles 
of $G$
\end{lemma}
\begin{proof}
Suppose that $z$ has degree at least~$4k$. 
Let $G'$ be the graph arising from subdividing every edge incident with $z$ once.
By~\eqref{G2con}, $G$ is $2$-connected, and $G'$ therefore too. 
Thus, there is a tree $T'$ in $G'-z$ that contains all neighbours of $z$. Of these
at least $4k$ neighbours, at least $2k$ lie in the same partition class of the bipartite
graph $T'$; let the set of these neighbours be $Z$. Lemma~\ref{Asublem} yields a set $\mathcal P'$ 
of $k$
edge-disjoint $Z$-paths in $T'$, each of which has even length by choice of $Z$. 
Moreover, since the vertices in $Z$ have degree~$1$ in $G'$, no vertex in $Z$ lies in two
paths of $\mathcal P'$.
Now, by identifying the endvertices of the paths in $\mathcal P'$
with $z$, we obtain $k$ edge-disjoint even cycles in $G$ that each pass through $z$. 
As $z\in A$
we have thus found $k$ edge-disjoint even $A$-cycles, which is impossible by~\eqref{nok}.

Therefore $z$ has degree less than $4k$. By~\eqref{singlez},
the set $F$ of edges incident with $z$ meets all even $A$-cycles.
\end{proof}

Thus, from now on, we may assume that
\begin{equation}\label{znotA}
z\notin A.
\end{equation}

\begin{lemma}\label{blockdeglem}
Let $B$ be a block of $G-z$, and let $a\in A\cap V(B)$. Then $a$ has at most two neighbours in $B$.
\end{lemma}
\begin{proof}
Suppose that $a$ has three neighbours $b_1,b_2,b_3$ in $B$. Pick a smallest tree $T$ in $B-a$ that 
contains $b_1,b_2,b_3$. Then $T$ together with $a$ and the edges $ab_1,ab_2,ab_3$ forms a theta graph $\theta$
with branch vertex $a$. By Lemma~\ref{thetalem}, $\theta$ contains an even $A$-cycle that then is disjoint
from $z$, a contradiction to~\eqref{singlez}.
\end{proof}

For a subgraph $H$ of $G$,  an even cycle is \emph{$H$-heavy}
if it passes through an edge of $H$ that is incident with a vertex in $A$. 
Observe that, since $z\notin A$ by~\eqref{znotA},
every even $A$-cycle needs to be $B$-heavy 
for some block $B$ of $G-z$.

By~\eqref{G2con}, $G$ is $2$-connected, which means that
$G-z$ is connected and therefore admits a block-tree.
It also follows that
\begin{equation}\label{leafblock}
\emtext{
every leaf-block of the block-tree of $G-z$ contains a neighbour of $z$.
}
\end{equation}

\begin{lemma}\label{blockhitlem}
For every block $B$ of $G-z$ there is a set $F_B\subseteq E(G)$ of size $|F_B|\leq 12k$ 
that meets every $B$-heavy even cycle. 
\end{lemma}
\begin{proof}
Assume first that $B$ contains at most two vertices from $A$. Then let $F_B$ be the set of edges of $B$
that are incident with vertices from $A$. By Lemma~\ref{blockdeglem}, we have $|F_B|\leq 4$. 
Clearly, $F_B$ meets every even $A$-cycle that is $B$-heavy. 

Thus, we assume now that $B$ contains at least three vertices from $A$.
Consider a cycle $C\subseteq B$ through at least one vertex of $A$. Then, every vertex in $A\cap V(B)$
lies in $C$; otherwise we could find a $C$-path $P$ in $B$ that passes through some vertex $a'\in V(B-C)$,
which would imply with Lemma~\ref{thetalem}
that the theta graph $C\cup P$ contained an even $A$-cycle that avoids  $z$, which is impossible by~\eqref{singlez}.

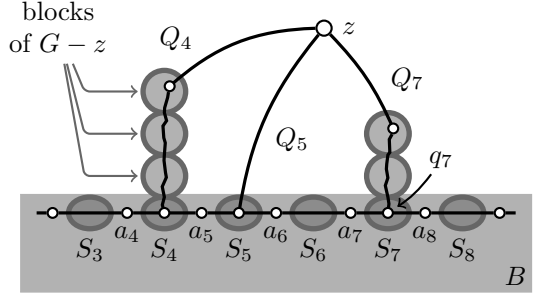
\begin{figure}[ht]
\centering
\begin{tikzpicture}[scale=0.7]

\def\xrad{0.4}
\def\yrad{0.3}
\def\crad{0.4}
\def\step{0.3}

\newcommand{\Si}[1]{
\begin{scope}[shift={#1}]
\draw[line width=2pt,dunkelgrau,fill=grauish] (0,0) ellipse [x radius=\xrad, y radius=\yrad];
\end{scope}
}

\draw[hellgrau,fill=hellgrau] ((-3*\xrad-4*\step-0.3,\yrad+0.05) rectangle (9*\xrad+10*\step+0.3,-1.5);
\node at (9*\xrad+10*\step,-1.2) {$B$};

\Si{(0,0)}
\Si{(2*\xrad+2*\step,0)}
\Si{(4*\xrad+4*\step,0)}
\Si{(6*\xrad+6*\step,0)}
\Si{(8*\xrad+8*\step,0)}
\Si{(-2*\xrad-2*\step,0)}

\foreach \i in {3,4,...,8}{
  \node at (2*\i*\xrad+2*\i*\step-8*\xrad-8*\step,-0.7) {$S_{\i}$};
}

\foreach \i in {4,...,8}{
  \node at (2*\i*\xrad+2*\i*\step-9*\xrad-9*\step,-0.4) {$a_{\i}$};
}

\block{(0,\yrad+\crad)}
\block{(0,\yrad+3*\crad)}
\block{(0,\yrad+5*\crad)}

\block{(6*\xrad+6*\step,\yrad+\crad)}
\block{(6*\xrad+6*\step,\yrad+3*\crad)}

\node[tinyvx] (aa) at (-3*\xrad-3*\step,0){};
\node[tinyvx] (a0) at (-\xrad-\step,0){};
\node[tinyvx] (a1) at (\xrad+\step,0){};
\node[tinyvx] (a2) at (3*\xrad+3*\step,0){};
\node[tinyvx] (a3) at (5*\xrad+5*\step,0){};
\node[tinyvx] (a4) at (7*\xrad+7*\step,0){};
\node[tinyvx] (a5) at (9*\xrad+9*\step,0){};

\draw[edg] (aa) -- (a0) -- (a1) -- (a2) -- (a3) -- (a4) -- (a5);
\draw[edg] (aa) edge ++(-\step,0);
\draw[edg] (a5) edge ++(\step,0);

\node[vx,label=right:$z$] (z) at (3,3.5){};
\draw[edg,bend right=20] (z) to node[auto,swap,near end] {$Q_4$} (0.1,\yrad+5*\crad+0.1) node[tinyvx] (z1){};
\draw[edg,bend left=10] (z) to node[auto,near end] {$Q_7$} (6*\xrad+6*\step+0.1,\yrad+3*\crad+0.1) node[tinyvx] (z2){};
\draw[edg,bend right=15] (z) to node[auto,midway] {$Q_5$} (2*\xrad+2*\step,0) node[tinyvx] (z3){};

\draw[patedg] (z1) to (0,\yrad+4*\crad) to (0,\yrad+2*\crad) to (0,\yrad) to (0,0);
\draw[patedg] (z2) to (6*\xrad+6*\step,\yrad+2*\crad) to (6*\xrad+6*\step,\yrad) to (6*\xrad+6*\step,0);

\node[align=center] (bl) at (-2,3.5) {blocks\\of $G-z$};
\draw[pointer,rounded corners=2pt] (bl) to (-1.5,\yrad+5*\crad) to (-\crad,\yrad+5*\crad);
\draw[pointer,rounded corners=2pt] (bl) to (-1.5,\yrad+3*\crad) to (-\crad,\yrad+3*\crad);
\draw[pointer,rounded corners=2pt] (bl) to (-1.5,\yrad+1*\crad) to (-\crad,\yrad+1*\crad);

\node[tinyvx] (q1) at (0,0){};
\node[tinyvx] (q3) at (6*\xrad+6*\step,0){};

\node[inner sep=2pt] (Q3) at (6*\xrad+6*\step+1,1){$q_7$};
\draw[pointer,bend left=10,black] (Q3) to (q3);

\end{tikzpicture}
\caption{A theta graph in the proof of Lemma~\ref{blockhitlem}}\label{blockfig}
\end{figure}

Let $a_0,a_1,\ldots, a_{\ell-1}$ be an enumeration of $A\cap V(B)$ in the order in which the vertices appear on $C$.
Note that $\ell\geq 2$. Then, $C$ splits into edge-disjoint subpaths $P_0,\ldots,P_{\ell-1}$ such that $P_i$
has endvertices $a_i$ and $a_{i+1}$, where indices are taken mod~$\ell$.
For $i=0,\ldots, \ell-1$ 
define $S_i$ to be the union of $P_i$ with all $C$-paths with both endvertices in $P_i$. 

We prove:
\begin{enumerate}[\rm (i)]
\item\label{bluba} $B=\bigcup_{i=0}^{\ell-1}S_i$;
\item\label{blubb} $S_i\cap S_j=\emptyset$ whenever $i+1<j$; and
\item\label{blubc} $S_i\cap S_{i+1}=\{a_{i+1}\}$ for every $i=0,\ldots, \ell-1$ (indices taken mod~$\ell$).
\end{enumerate} 
Indeed, for~\eqref{bluba} observe that no $C$-path in $B$ may meet $a_0,\ldots, a_{\ell-1}$ as these 
vertices have degree~$\leq 2$ in $B$, by Lemma~\ref{blockdeglem}, and each of the vertices
has already degree~$2$ in $C$. Thus, any $C$-path $Q$ with its endvertices not both in some $P_i$
has its endvertices in the interior of two distinct $P_i,P_j$ with $i<j$. Then $C\cup Q$
is a theta graph such that two of its subdivided  edges meet one of $a_i,a_{i+1}$. 
Lemma~\ref{thetalem} yields again an even $A$-cycle that avoids $z$, which is impossible by~\eqref{singlez}.

For~\eqref{blubb} and~\eqref{blubc}, suppose there were $i<j$ and a $C$-path $Q_i$ with its endvertices in $P_i$, 
and a $C$-path $Q_j$ with its endvertices in $P_j$ such that $Q_i$ and $Q_j$ meet, in a vertex $v$ say. 
Then $v$ must be an interior vertex of $Q_i$ (and of $Q_j$). Thus, $Q_ivQ_j$ is a $C$-path with its endvertices
in distinct paths of $P_0,\ldots, P_{\ell-1}$, which is impossible by~\eqref{bluba}.

Next:
\begin{equation}\label{fewcut}
\begin{minipage}[c]{0.8\textwidth}\em
Fewer than $3k$ of the graphs  $S_{0}-a_{1},\ldots, S_{\ell-1}-a_{\ell}$ contain  
a cutvertex of $G-z$ or a neighbour of $z$.
\end{minipage}\ignorespacesafterend 
\end{equation} 
Suppose that for $i_0<\ldots <i_{3k-1}$ each of $S_{i_j}-a_{i_j+1}$ contains a cutvertex of $G-z$
or a neighbour of $z$. In each case, we find, by~\eqref{leafblock}, a $C$--$z$~path $Q_{i_j}$ 
that starts in $P_{i_j}$ and meets $B$ only in $S_{i_j}$. In particular, the paths $Q_{i_j}$ pairwise
meet only in $z$. Let $q_{i_j}$ be the endvertex of $Q_{i_j}$ on $C$, and denote by $R_{i_j}$ 
the $q_{i_j}$--$q_{i_{j+1}}$~subpath of $C$ that contains $a_{i_j+1}$. Then for $t=0,\ldots, 3k-1$
\[
\theta_t=Q_{i_{3t}}\cup R_{i_{3t}} \cup Q_{i_{3t+1}} \cup R_{i_{3t+1}}\cup Q_{i_{3t+2}}
\]
is a theta graph such that two of its subdivided edges, namely $Q_{i_{3t}}\cup R_{i_{3t}}$
and $R_{i_{3t+1}}\cup Q_{i_{3t+2}}$, meet $A$ (in $a_{i_{3t}+1}$, resp.\ in $a_{i_{3t+1}+1}$). 
Thus, $\theta_0,\ldots,\theta_{k-1}$
are edge-disjoint graphs that each, by Lemma~\ref{thetalem}, contain an even $A$-cycle
in contradiction of~\eqref{nok}. This proves~\eqref{fewcut}.

Let $i_0<\ldots <i_s$ be the sequence of all indices $i_j$ such that $S_{i_j}-a_{i_j+1}$
contains a cutvertex of $G-z$ or a neighbour of $z$; and let $F_B$ be the set of edges
of $B$ 
incident with $\{a_{i_j},a_{i_j+1}:j=0,\ldots, s\}$. Then, as each $a_{i_j}$ has degree~$2$
in $B$, it follows from~\eqref{fewcut} that $|F_B|\leq 2\cdot 2\cdot 3k$. 

Consider a  $B$-heavy even $A$-cycle $D$. Then, 
$D$ contains an edge $ab\in E(B)$ such that $a\in A$. 
As $D$ passes through $z$,
by~\eqref{singlez}, it also contains a subpath $R$ that is completely contained in $B$,
that starts in  a cutvertex of $G-z$ or in a neighbour of $z$ in $B$
and that contains $ab$.
Let the endvertex $r\neq a$ of $R$ be in $S_{i_j}$. Now either $a\in \{a_{i_j},a_{i_j+1}\}$, in which case it holds
that $ab\in F_B$, or $a$ lies outside $S_{i_j}$. Then $R$ needs to pass through one of $\{a_{i_j},a_{i_j+1}\}$
and thus through an edge in $F_B$. In every case, $D$ meets $F_B$.
\end{proof}

Consider a
 path $S=b_0B_1\ldots B_{\ell}b_{\ell}$ in the block-tree of $G-z$, 
where $B_1,\ldots,B_\ell$ are blocks and $b_0,\ldots, b_\ell$ are cutvertices of $G-z$.
The path $S$ is a \emph{string} if
\begin{itemize}
\item $z$ does not have any neighbour in $\bigcup_{i=1}^\ell B_i$;
\item every $B_i$ and every $b_j$ has degree~$2$ in the block-tree;  
\item $B_1$ and $B_\ell$ contain a vertex from $A$; and 
\item the path $S$ is maximal subject to the other conditions. 
\end{itemize}

I will sometimes treat a string $S$ also as a subgraph of $G-z$, that is, 
I will speak of a string $S=b_0B_1\ldots B_{\ell}b_{\ell}$ but mean 
the subgraph $\bigcup_{i=1}^\ell B_i$. In particular, when I consider an even $A$-cycle 
and it is $S$-heavy, I will mean that the cycle is 
$\bigcup_{i=1}^\ell B_i$-heavy. 

\begin{lemma}\label{stringlem}
For every string $S$ there is an edge set $F_S\subseteq E(G)$ with $|F_S|\leq 24k^2$
such that $F_S$ meets every $S$-heavy even $A$-cycle. 
\end{lemma}
\begin{proof}
Let $S=b_0B_1\ldots B_{\ell}b_{\ell}$.
First assume that fewer than $2k$ of the blocks $B_i$ contain a vertex from $A$.
For each such block $B_i$ that meets $A$, let $F_{B_i}$ be as in Lemma~\ref{blockhitlem},
i.e.\ an edge set of at most $12k$ edges that meets every $B_i$-heavy even $A$-cycle. 
For a block $B_i$ that is disjoint from $A$ put $F_{B_i}=\emptyset$. Then 
\[
F_S=\bigcup_{i=1}^\ell F_{B_i}
\]
meets every $S$-heavy even $A$-cycle. Moreover, it has size $|F_S|\leq 2k\cdot 12k=24k^2$.

Thus, assume that at least $2k$ of the blocks $B_i$ contain a vertex from $A$. 
Next, if possible, choose an $F_S\subseteq E(G)$ of size $|F_S|\leq 10k$ such that
\begin{itemize}
\item $F_S$ separates $z$ from $b_0$ in $G-(S-b_0)$; or 
\item $F_S$ separates $z$ from $b_\ell$ in $G-(S-b_\ell)$; or
\item $F_S$ separates $b_0$ from $b_\ell$ in $S$.
\end{itemize}
As any $S$-heavy even $A$-cycle  would pass through any such separator,
we would be done if there was such an $F_S$. Thus, we may assume that is not the case. 
Then, in particular, there is a set $\mathcal C$ of  $10k$ 
edge-disjoint cycles that pass through $z$, through each of $b_0,b_1,\ldots, b_\ell$
and through an edge from each of $B_1,\ldots, B_\ell$.

Pick indices $i_1<\ldots <i_{2k}$ such that the blocks $B_{i_j}$ each contain 
a vertex $a_{i_j}\in A$. 
%Observe that none of $a_{i_1},\ldots,a_{i_{2k}}$
%can be a cutvertex of $G-z$ as, by Lemma~\ref{blockdeglem}, such a vertex would \mymargin{cutvx - need?}
%have degree at most~$2$ in its blocks (which would mean that $b_0$ could be separated from $b_\ell$
%by at most two edges).
%
Next, we note that,  because, by Lemma~\ref{blockdeglem}, the vertices $a_{i_j}$ have degree at most (in fact, 
exactly)~$2$ in $B_{i_j}$,
\begin{equation}\label{onlyone}
\emtext{
 each $a_{i_j}$ may lie in at most one of the cycles in $\mathcal C$.
% we need this for the theta later on -- the second vx in A should lie outside the cycle to form a theta
}
\end{equation}

\begin{figure}[ht]
\centering
\begin{tikzpicture}
\draw[edg,dunkelgrau] (0,0) arc [start angle=165, end angle=15, radius=0.3]
arc [start angle=195, end angle=360-15, radius=0.3]
arc [start angle=165, end angle=15, radius=0.3]
arc [start angle=195, end angle=360-15, radius=0.3]
arc [start angle=165, end angle=15, radius=0.3]
arc [start angle=195, end angle=360-15, radius=0.3]
arc [start angle=165, end angle=15, radius=0.3]
node[tinyA,near end] (a2){}
arc [start angle=195, end angle=360-15, radius=0.3]
arc [start angle=165, end angle=15, radius=0.3]
;

\draw[edg] (0,0) 
node[tinyvx,label=below:$b_0$](b1){} 
arc [start angle=195, end angle=360-15, radius=0.3]
node[tinyvx,label=below:$b_1$]{}
arc [start angle=165, end angle=15, radius=0.3]
node[tinyA,midway,label=above:$a_1$]{}
arc [start angle=195, end angle=360-15, radius=0.3]
arc [start angle=165, end angle=15, radius=0.3]
node[tinyvx,label=below:$b_2$]{}
arc [start angle=195, end angle=360-15, radius=0.3]
coordinate[midway] (zz)
arc [start angle=165, end angle=15, radius=0.3]
node[tinyvx,label=below:$b_3$]{}
arc [start angle=195, end angle=360-15, radius=0.3]
arc [start angle=165, end angle=15, radius=0.3]
node[tinyvx,label=below:$b_4$]{}
arc [start angle=195, end angle=360-15, radius=0.3]
node[tinyvx,label=below:$b_\ell$](bl){}
;

\path (zz) to ++(0,2) node[vx,label=above:$z$] (z){};
\path (b1) to ++(-0.5,0.6) coordinate (w1);
\path (bl) to ++(0.5,0.6) coordinate (wl);
\draw[edg,dunkelgrau] (b1) to [out=225,in=-90] (w1) to [out=90,in=180] (z);
\draw[edg,dunkelgrau] (bl) to [out=-45,in=-90] (wl) to [out=90,in=0] (z);
\draw[edg,out=105,in=190] (b1) to node[near end, auto,swap,inner sep=2pt] {$C_1$}(z);
\draw[edg,out=90-15,in=-10] (bl) to (z);

\node[label=left:$C_2$] at (w1) {};
\node[label=above:$a_2$] at (a2) {};

\end{tikzpicture}
\caption{$C_1\cup C_2$ as in the proof of Lemma~\ref{stringlem} contains a theta graph}\label{twocycfig}
\end{figure}
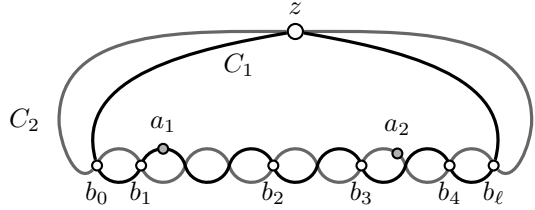

Next, consider an $a_{i_j}$ that does not lie in any cycle of $\mathcal C$. 
Then there is a path $P\subseteq B_{i_j}$ through $a_{i_j}$ that starts in  in some cycle
$C\in\mathcal C$, ends in some (possibly different) cycle $C'\in\mathcal C$ and is otherwise
edge-disjoint from all cycles in~$\mathcal C$.
With Lemma~\ref{abclem} we see that there is a $b_{{i_j}-1}$--$a_{i_j}$--$b_{i_j}$~path
in $C\cup C'\cup P$. Using this path to replace $C\cap B_{i_j}$ in $C$, we obtain  
a new cycle $\tilde C$ that passes through $a_{i_j}$ and that is edge-disjoint from 
every cycle in $\mathcal C$, except for $C$ and $C'$. 
If $C'$ is disjoint from $a_{i_1},\ldots, a_{i_{2k}}$ then remove $C,C'$ from $\mathcal C$ and
add $\tilde C$.
If $C'$ already meets one 
of $a_{i_1},\ldots, a_{i_{2k}}$ then  choose some cycle $D\in\mathcal C$ that is disjoint from 
$a_{i_1},\ldots, a_{i_{2k}}$ (this is possible, by~\eqref{onlyone} and $|\mathcal C|\geq 10k$), and 
use $D\cap B_{i_j}$ to replace $C'\cap B_{i_j}$ in $C'$. We add the resulting cycle to $\mathcal C$
as well as $\tilde C$, and remove $C,C'$ and $D$ from $\mathcal C$. 
In both cases, the new set $\mathcal C$ will consist of pairwise edge-disjoint cycles of a size one less
than before. Moreover, one more of $a_{i_1},\ldots, a_{i_{2k}}$ will lie in a cycle of $\mathcal C$.
By repeating this process, we will find a set $C_1,\ldots, C_{2k}$ of  pairwise disjoint cycles 
each passing through $z$
such that each of $a_{i_1},\ldots, a_{i_{2k}}$ lies in exactly one cycle of $\mathcal C$ 
(we use~\eqref{onlyone} here).

Then, however, each of $C_1\cup C_2$, $C_3\cup C_4$, $\ldots,$ $C_{2k-1}\cup C_{2k}$ contains 
a theta graph such that in each at least two subdivided edges meet $A$; see Figure~\ref{twocycfig}. Again, we find with 
Lemma~\ref{thetalem} $k$ edge-disjoint even $A$-cycles, which is impossible by~\eqref{nok}.
\end{proof}

\newcommand{\nots}{\mathcal A_{\overline{\mathcal S}}}

Let $\mathcal S$ be the set of strings. 
Let $\nots$ be the set of blocks of $G-z$
that meet $A$ and that are not contained in any string. We note that:
\begin{equation}\label{Yheavy}
\emtext{
every even $A$-cycle is $Y$-heavy for some $Y\in\mathcal S\cup \nots$. 
}
\end{equation}
Indeed, every $A$-cycle $C$ is $B$-heavy for some block $B$ of $G-z$. 
If $B\notin \nots$ then $B$ lies in some string $S$, which implies
that $C$ is $S$-heavy.

For each string $S\in\mathcal S$, pick a block in it 
that meets $A$ and denote by $\mathcal A_{\mathcal S}$ the set of 
all the chosen blocks. Note that 
\[
|\mathcal S|=|\mathcal A_{\mathcal S}|
\]
as strings are disjoint.

Define $\mathcal T'$ to be the tree obtained from the block-tree of $G-z$ by 
iteratively deleting all leaves that are not in $\mathcal A_{\mathcal S}\cup\nots$ and by iteratively suppressing
vertices of degree~$2$ that are not in $\mathcal A_{\mathcal S}\cup\nots$. 
Clearly, $\mathcal T'$ contains all of $\mathcal A_{\mathcal S}\cup\nots$, and every leaf of $\mathcal T'$
and every vertex of degree~$2$ lies in $\mathcal A_{\mathcal S}\cup\nots$. 

We will now go through the different outcomes of Lemma~\ref{alttreelem}.

\begin{lemma}\label{clawlem}
$\mathcal T'$ does not contain $k$ disjoint subtrees that each contain three leaves of $\mathcal T'$.
\end{lemma}
\begin{proof}
Suppose there are such subtrees. Then there are also $k$ disjoint subtrees $T_1,\ldots, T_k$
of the block-tree of $G-z$ such that:
\begin{enumerate}[\rm (i)]
\item $T_i$ has exactly three leaves $L_1,L_2,L_3$, and these are leaf-blocks of the block-tree of $G-z$; and
\item\label{clawb} if $B_i$ is the degree~$3$-vertex in $T_i$ then each of the $B_i$--$L_j$~paths in $T_i$
contains  a block, not $B_i$, that contains a vertex from $A$.
\end{enumerate}
For~\eqref{clawb}, note that the leaves of $\mathcal T'$ lie in $\mathcal A_{\mathcal S}\cup\nots$.

Consider  a $T_i$. Then, for $j=1,2,3$, there is, by~\eqref{leafblock}, a $B_i$--$z$~path $Q_j\subseteq G$ 
contained in the union of $z$ with the blocks of $G-z$ in the $B_i$--$L_j$~path
of $T_i$. 
Moreover, because of~\eqref{clawb} and Lemma~\ref{abclem} 
we can ensure that each $Q_j$ passes through some vertex
in $A$.  Let $S$ be a tree in $B_i$ that contains 
the endvertices of $Q_1,Q_2,Q_3$. Then $S\cup Q_1\cup Q_2\cup Q_3$ contains a theta graph $\theta_i$ 
such that each of its subdivided edges passes through some vertex in $A$ (as $Q_1,Q_2,Q_3$ pass through $A$). As
$\theta_1,\ldots,\theta_k$ are pairwise edge-disjoint and as, by Lemma~\ref{thetalem}, each $\theta_i$ contains 
an even $A$-cycle, we have again obtained $k$ edge-disjoint even $A$-cycles in contradiction to~\eqref{nok}. 
\end{proof}

\begin{lemma}\label{longpathlem}
$\mathcal T'$ does not contain any path of length~$15k$.
\end{lemma}
\begin{proof}
Suppose that $\mathcal T'$ contains a path $\mathcal P'$ of length at least $15k$.
We first treat the case when 
$\mathcal P'$ 
contains at least $3k$ vertices that have degree at least~$3$ in $\mathcal T'$.
By always grouping  three consecutive such vertices together with their branches hanging off
${\mathcal P'}$ we  find $k$  disjoint subtrees of  
$\mathcal T'$ that each contain three leaves of $\mathcal T'$. This, however, is impossible by Lemma~\ref{clawlem}.

Therefore, $\mathcal P'$ needs to contain at least $12k$ vertices of degree~$2$. By definition
of $\mathcal T'$ each such vertex lies in $\mathcal A_{\mathcal S}\cup\nots$.
Observe that
then there is also  a path $\mathcal P$ in the block-tree of $G-z$ that contains $12k$ vertices that lie in $\mathcal A_{\mathcal S}\cup\nots$. 
We partition $\mathcal P$ 
into edge-disjoint subpaths $\mathcal P_1,\ldots, \mathcal P_{3k}$ such that each subpath $\mathcal P_s$
contains at least four vertices from $\mathcal A_{\mathcal S}\cup\nots$. 

\begin{figure}[ht]
\centering
\begin{tikzpicture}
\def\crad{0.4}

\foreach \i in {0,...,10}{
\block{(2*\crad*\i,0)}
}

\node at (4*\crad,-0.7) {$B_{i_1}$};
\node at (8*\crad,-0.7) {$B_{i_2}$};
\node at (12*\crad,-0.7) {$B_{i_3}$};
\node at (18*\crad,-0.7) {$B_{i_4}$};

\draw[patedg] (-\crad,0) to ++(2*\crad,0) to ++(2*\crad,0) to ++(2*\crad,0) to ++(2*\crad,0) to ++(2*\crad,0)
 to ++(2*\crad,0) to ++(2*\crad,0) to ++(2*\crad,0) to ++(2*\crad,0) to ++(2*\crad,0) to ++(2*\crad,0); 

\node[tinyvx] (a1) at (4*\crad-0.1,0){};
\node[tinyvx] (a2) at (18*\crad-0.05,0){};
\node[tinyvx,label=left:$b_0$] (b0) at (-\crad,0){};
\node[tinyvx,label=right:$b_\ell$] (bl) at (21*\crad,0){};

\node[vx,label=right:$z$] (z) at (10*\crad,2){};
\node[tinyvx,label=below right:$x$] (x) at (9*\crad+0.2,0){};
\draw[edg,bend right=10] (z) to node[auto,swap,midway] {$Q$} (x);

\node[inner sep=0cm] (P) at (2,1){$P$};
\draw[pointer,bend left=20] (P) to (6*\crad,0);

\node[inner sep=2pt] (A1) at (1.2,1){$\in A$};
\draw[pointer,bend left=20] (A1) to (a1);

\node[inner sep=2pt] (A2) at (7,1){$\in A$};
\draw[pointer,bend left=20] (A2) to (a2);

\draw[thick,dunkelgrau] (-\crad,-1) -- ++(0,-0.1) to coordinate[midway] (m) ++(22*\crad,0) -- ++(0,0.1);
\draw[thick,dunkelgrau] (m) to node[below,text=black] {$\mathcal P_s$} ++(0,-0.1);

\node[align=center,inner sep=2pt] (B) at (14*\crad,1.8) {blocks\\of $G-z$};
\draw[pointer,bend right=20] (B) to (12*\crad,\crad);
\draw[pointer] (B) to (14*\crad,\crad);
\draw[pointer,bend left=20] (B) to (16*\crad,\crad);

\end{tikzpicture}
\caption{$R_s$ in the proof of Lemma~\ref{longpathlem}}\label{Rfig}
\end{figure}
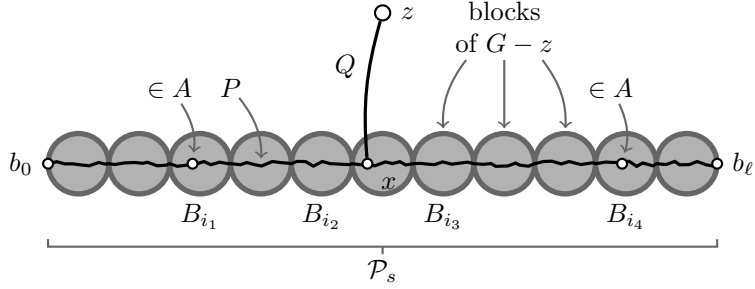

Consider an arbitrary subpath $\mathcal P_s=b_0B_1b_1\ldots b_{\ell-1}B_\ell b_\ell$. 
Let $B_{i_1}$, $B_{i_2}$, $B_{i_3}$, $B_{i_4}$ with $i_1<i_2<i_3<i_4$ be blocks 
in $\mathcal A_{\mathcal S}\cup\nots$. 
Pick a $b_0$--$b_\ell$~path $P$ that passes through every block $B_i$. Moreover, 
we can ensure that $P$ passes through a vertex of $A$ in $B_{i_1}$ and in $B_{i_4}$
(since the $B_i$ are either a single edge or $2$-connected).

Note that
$B_{i_2}$ and $B_{i_3}$ do not lie in a common string:
indeed, if $B_{i_2},B_{i_3}\in\mathcal A_{\mathcal S}$ then they lie in distinct strings, and 
distinct strings are disjoint.
As a consequence, the subpath $b_{i_2-1}B_{i_2}\ldots B_{i_3}b_{i_3}$ of $\mathcal P_s$
must contain either a vertex that has degree~$\geq 3$ in the block-tree or some 
 neighbour of $z$. Thus, there is, by~\eqref{leafblock}, a $P$--$z$~path $Q$ that starts in some vertex $x\in \bigcup_{j=i_2}^{i_3}B_j$.
Note that each subpath $b_0Px$ and $xPb_\ell$ contains a vertex of $A$, namely some vertex in $B_{i_1}$ and in $B_{i_4}$. 
Denote $P\cup Q$ by $R_s$, and note that $R_1,\ldots, R_{3k}$ are pairwise edge-disjoint.

Now, we observe that each of $R_1\cup R_2\cup R_3$,\dots, $R_{3k-2}\cup R_{3k-1}\cup R_{3k}$ 
contains a theta graph such that two of its subdivided edges pass through $A$. By Lemma~\ref{thetalem},
we thus find $k$ edge-disjoint even $A$-cycles, which we had excluded~\eqref{nok}.
\end{proof}

\begin{lemma}\label{starlem}
$\mathcal T'$ does not contain a vertex of degree at least~$3k$.
\end{lemma}
\begin{proof}
Suppose there is such a vertex. Then, clearly, there is also a vertex $X$ in the block-tree of $G-z$
of degree at least $3k$. This implies that there are $3k$ paths $\mathcal P_1,\ldots, \mathcal P_{3k}$
in the block-tree between $X$ and the leaves of the block-tree such that pairwise the paths meet
exactly in $X$. Moreover, as the leaves of $\mathcal T'$ are a subset of $\mathcal A_{\mathcal S}\cup\nots$, we can 
choose the paths $\mathcal P_i$ such that each contains a block different from $X$ that contains a 
vertex from $A$.

We now pick for each $i$ a $X$--$z$~path $P_i\subseteq G$ 
such that $P_i-z$ is  contained in 
$\bigcup_{B\in\mathcal P_i}B$; this is possible because of~\eqref{leafblock}.
Since $\mathcal P_i$ contains a block different from $X$ that meets $A$, we can 
furthermore ensure that $P_i$ passes through $A$ (using Lemma~\ref{abclem} if necessary).

Seen in $G-z$, 
the vertex $X$ in the block-tree is either a block or a cutvertex. Extend $X$ to a graph $X'$ by
adding a new vertex $p_i$ for each $P_i$ and by connecting it to the endvertex of $P_i$ in $X$.
We apply Lemma~\ref{Asublem} to a spanning tree of $X'$, with the set $\{p_1,\ldots, p_{3k}\}$
in the role of $Z$. Thus, there are $k$ edge-disjoint trees $S_1,\ldots, S_k\subseteq X'$
that each contain three of the vertices in $\{p_1,\ldots, p_{3k}\}$. By changing the enumeration, 
we may assume that $S_1$ contains $p_1,p_2,p_3$, that $S_2$ contains $p_4,p_5,p_6$ and so on.

Now, for $i=1,\ldots, k$, the graph 
$S_i-\{p_{3i-2},p_{3i-1},p_{3i}\}\cup P_{3i-2}\cup P_{3i-1}\cup P_{3i}$
contains a theta graph $\theta_i$ such that all three of its subdivided edges pass through a vertex from $A$.
The $\theta_i$ are pairwise edge-disjoint and each contains, by Lemma~\ref{thetalem}, an even $A$-cycle ---
this is again a contradiction to~\eqref{nok}.
\end{proof}

\begin{proof}[Proof of Theorem~\ref{evenathm}]
Suppose first that $|\mathcal A_{\mathcal S}\cup\nots|\geq 45k^3$. Since $V(\mathcal A_{\mathcal S}\cup\nots)\subseteq V(\mathcal T')$
we may apply Lemma~\ref{alttreelem} with $\alpha=3k$, $\beta=k$ and $\gamma=15k$.
Each of the outcomes of Lemma~\ref{alttreelem}, however, leads to a contradiction as 
Lemmas~\ref{clawlem},~\ref{longpathlem} and~\ref{starlem} demonstrate. 

We conclude that $|\mathcal A_{\mathcal S}\cup\nots|< 45k^3$. Then $|\mathcal S\cup\nots|< 45k^3$. 
For each $Y\in \mathcal S\cup\nots$ we pick $F_Y$ as in Lemma~\ref{blockhitlem} or~\ref{stringlem}, depending
on whether $Y\in\mathcal S$ or $Y\in\nots$. In each case $|F_Y|\leq 24k^2$, which means that 
\[
F=\bigcup_{Y\in\mathcal S\cup\nots}F_Y\emtext{ has size }|F|\leq 45k^3\cdot 24 k^2=1080k^5
\]
Consider some even $A$-cycle $D$. Then, by~\eqref{Yheavy},
 $C$ must be $Y$-heavy for some $Y\in\mathcal S\cup\nots$
and thus meets $F_Y\subseteq F$. 
This shows that $F$ meets every even $A$-cycle. 
\end{proof}

\subsection*{Acknowledgement}
I thank Arthur Ulmer for greatly improving the construction in Proposition~\ref{modprop}.

\bibliographystyle{amsplain}
\bibliography{../erdosPosa/erdosposa}

\vfill

\small
\vskip2mm plus 1fill
\noindent
Version \today{}
\bigbreak

\noindent
Henning Bruhn-Fujimoto
{\tt <henning.bruhn@uni-ulm.de>}\\
Institut f\"ur Optimierung und Operations Research\\
Universit\"at Ulm\\
Germany\\

\end{document}